\documentclass[12pt,a4paper,twoside]{amsart}
\usepackage{amsfonts, amsthm, amsmath, amssymb}
\usepackage{mathrsfs,amsmath}
\usepackage{hyperref}
\usepackage{bbm}
\hypersetup{colorlinks=false}

\usepackage[margin=2.9cm]{geometry}

\usepackage{helvet}

\usepackage{amsthm}
\usepackage{lineno}
\newtheorem{theorem}{Theorem}
\newtheorem{lemma}{Lemma}[section]
\newtheorem{remark}{Remark}

\newtheorem{proposition}{Proposition}

\begin{document}

\author{ Aritra Ghosh }
\title{Weyl-type bounds for twisted GL(2) short character sums }

\address{Aritra Ghosh \newline  Stat-Math Unit, Indian Statistical Institute, 203 B.T. Road, Kolkata 700108, India; email: aritrajp30@gmail.com}

\begin{abstract}
Let $f$ be a Hecke-Maass or holomorphic primitive cusp form of full level for $SL(2,\mathbb{Z})$ with normalized Fourier coefficients $\lambda_{f}(n)$. Let $\chi$ be a primitive Dirichlet character of modulus $p$, a prime. In this article we shorten the range of cancellation for $N$ in the twisted $GL(2)$ short character sum. Here we consider the problem of cancellation in short character sum of the form 
$$S_{f,\chi}(N):= \mathop\sum_{n \in \mathbb{Z}}\lambda_{f}(n)\chi(n)W\left(\frac{n}{N}\right) .$$ 
We show that, for $0<\theta < \frac{1}{10}$,
$$S_{f,\chi}(N) \ll_{f,\epsilon}N^{3/4 + \theta /2}p^{1/6}(pN)^{\epsilon} + N^{1-\theta }(pN)^{\epsilon},$$
which is non-trivial if $N \geq p^{2/3 + \alpha + \epsilon}$ where $\alpha = = \frac{4\theta}{1-6\theta}$. Previously such a bound was known for $N \geq p^{3/4 + \epsilon}.$

\end{abstract}

\maketitle

\section{Introduction}
A problem which arises in a variety of contexts is the cancellation in sums of the form
\begin{equation} \label{eq3}
S_{\chi}(N)=\mathop\sum_{n \in \mathbb{Z}}\chi(n)W\left(\frac{n}{N}\right),
\end{equation}
and
\begin{equation} \label{eq4}
S_{f,\chi}(N) = \mathop\sum_{n \in \mathbb{Z}}\lambda_{f}(n)\chi(n)W\left(\frac{n}{N}\right),
\end{equation} 
where $\chi$ is a character of conductor $p$, $\lambda_{f}(n)$'s are normalized Fourier coefficients of a Hecke-Maass or holomorphic primitive cusp form $f$ for $SL(2,\mathbb{Z})$ and $W$ is a smooth bump function supported on $[1,2]$.

By applying the Mellin inversion formula
$$W(x) = \frac{1}{2 \pi i} \int_{(\sigma)} \tilde{W}(s) x^{-s} ds, \,  \sigma >1 ,$$
we see that the equation \eqref{eq3} becomes
\begin{equation}\label{eq5}
S_{\chi}(N) = \frac{1}{2\pi i} \mathop\int_{(\sigma)} N^{s} \widetilde{W}(s)L(s, \chi )ds,
\end{equation}
where $L(s,\chi )$ is the Dirichlet $L$-function. We can shift the contour to the central line $\sigma = \frac{1}{2}$. As the Mellin transform $\tilde{W}(s)$ decays rapidly on the vertical line, the main contribution to the integral comes from the points near the center $\sigma = \frac{1}{2}$.

For example plugging in a bound
$$| L(\frac{1}{2}+ it , \chi)| \ll p^\epsilon (2 + |t|)^A ,$$
we get 
$$S_{\chi}(N) \ll \sqrt{N}p^\epsilon.$$

In particular if we take the convexity bound
 $$L( 1/2 + it , \chi)\ll_{\epsilon} p^{1/4 + \epsilon}{(2+ |t|)}^{1/4 + \epsilon },$$
then we conclude that $S_{\chi}(N) \ll N^{1/2} p^{1/4} p^{\epsilon}$ which is non-trivial if and only if $N > p^{1/2 + \epsilon}$. Here one can note that the convexity bound recovers the conclusion of the Polya-Vinogradov inequality. Hence subconvexity corresponds to cancellation in shorter sums. D. A. Burgess (see \cite{B}) proved that $L( 1/2 , \chi )\ll_{\epsilon} p^{3/16 + \epsilon}$ which yields a non-trivial bound if and only if $N > p^{3/8 + \epsilon}$. Burgess's method required new ideas, in particular it uses the Riemann Hypothesis for curves over finite fields. Note that the Burgess exponent of $3/16$ falls short of the exponent $1/6$ found by H. Weyl. However Burgess's method yields a non-trivial bound for $S_{\chi}(N)$ for any $N \gg p^{1/4 +\epsilon}$ if $p$ is cube free (especially for primes). This does not come through the passage to $L$-functions as we have sketched above. But this basic idea is applicable to the scenarios as well, invoking higher rank harmonics. Recently I. Petrow-M. Young (see \cite{IP}) proved a Weyl-exponent subconvex bound for any Dirichlet $L$-function of cube-free conductor. They also got a bound of the same strength for certain $L$-functions of self-dual $GL(2)$ automorphic forms that arise as twists of forms of smaller conductor. One can  also see the recent work of P. Nelson (see \cite{PN}). Curiously, the exponent $3/16$ often re-occurs in the modern incarnations of these problems, see \cite{KRY}, \cite{BH}, \cite{BHM}, \cite{VB}, \cite{FK}, \cite{Wu1}, \cite{Wu2} as examples. Also related work on the Burgess type bounds can be found in the paper of R. Munshi (see \cite{Mun2}). For the case of Dirichlet $L$-functions, the Burgess bound has only been improved in some limited special cases. In a breakthrough, B. Conrey and H. Iwaniec (see \cite{CI}) obtained a Weyl quality bound for quadratic characters of odd conductor using techniques from automorphic forms and P. Deligne’s solution of the Weil conjectures for varieties over finite fields. Another class of results, such as  \cite{BLT} and \cite{HB}, consider situations where the conductor $q$ of $\chi$ runs over prime powers or otherwise has some special factorizations. Notably, D. Milic\'evi\'c (see \cite{MIL}) recently obtained a sub-Weyl subconvex bound when $q = p^n$ with $n$ large. Also for the Weyl bound for short twisted sums in the prime power case one can look the papers by V. Blomer and D. Mili\'cevi\'c (see \cite{ASE}) and a related paper of R. Munshi and S. K. Singh (see \cite{ANT}). Here a subconvex bound of the form $L(1/2, \pi) \ll_{\epsilon} Q(\pi)^{1/6 + \epsilon }$ is the Weyl bound, where $Q( \pi )$ is the analytic conductor of the automorphic $L$-function $L(1/2, \pi)$. The Weyl bound is only known in a few cases, notably for quadratic twists of certain self-dual $GL(2)$ automorphic forms; see \cite{CI}, \cite{IV}, \cite{PY1},\cite{Y1}.

 For the $GL(2)$ case we have the following:
 \begin{equation}\label{eq7}
 S_{f,\chi}(N) = \frac{1}{2\pi i} \mathop\int_{(\sigma)} N^{s} \widetilde{W}(s)L(s, f\otimes \chi )ds.
 \end{equation}
 Again by shifting the contour to the central line $\sigma = 1/2$, as $\tilde{W}(s)$ decays rapidly on the vertical line, the main contribution to the integral comes from the points near the contour $\sigma = 1/2$. Similarly plugging in a bound
 $$|L(\frac{1}{2}+ it, f \otimes \chi)|\ll p^\epsilon (3+ |t|)^A,$$
 we have
 \begin{equation}\label{eq8}
  S_{f,\chi}(N) \ll_{f,\epsilon} \sqrt{N}p^{\epsilon}.
 \end{equation}
 In this context, the convexity bound is $L(\frac{1}{2}+ it,f\otimes \chi )\ll_{f,\epsilon} p^{\frac{1}{2}+\epsilon}(3+|t|)^{\frac{1}{2}+\epsilon}$. So we have $S_{f,\chi}(N) \ll_{f,\epsilon} \sqrt{N}p^{\frac{1}{2}+\epsilon}$ which becomes non-trivial if and only if $N> p^{1+ \epsilon }$. Further improvement can be done. By the Burgess exponent (see \cite{VB}) we have 
 $$L(\frac{1}{2},f\otimes \chi ) \ll_{f,\epsilon} p^{3/8 + \epsilon}.$$
 Hence we have 
  $$S_{f,\chi}(N)\ll_{f,\epsilon} \sqrt{N}p^{\frac{3}{8} + \epsilon} < N \iff N > p^{3/4 + \epsilon}$$ 
  which may be called the Burgess range.

In this paper we will analyse the sum $S_{f,\chi}(N)$ using a version of $\delta$-method, without going into $L$-functions. Our method improves the range of cancellation from $N > p^{3/4 + \epsilon}$ (Burgess range) to $N > p^{2/3 +\epsilon}$ (Weyl range). Here we get the following result:

\begin{theorem}\label{th}
Let f be a Hecke-Maass or holomorphic primitive cusp of form full level for $SL(2,\mathbb{Z})$.  Let $\chi$ be a primitive Dirichlet character of modulus $p$, a prime. Then for any $\epsilon > 0$ and $0<\theta < \frac{1}{10}$ we have 
$$S_{f,\chi }(N)\ll_{f,\epsilon}N^{3/4 + \theta /2}p^{1/6}(pN)^{\epsilon} + N^{1-\theta }(pN)^{\epsilon},$$
which becomes non-trivial if  $p^{\frac{2}{3}+ \alpha + \epsilon}\leq N \leq p$, where $\alpha = \frac{4\theta}{1-6\theta}$.
\end{theorem}

$\hspace{-0.40 cm}$Though this result is implicit in the paper of R. Munshi (see \cite{Mun}), we are doing here explicitly. Actually in that paper (see \cite{Mun}) his aim is to get a subconvexity bound for $L(1/2 + it, f \otimes \chi)$ but here our aim is to get a range for $N$ to have a non-trivial bound or more precisely getting cancellation in our twisted $GL(2)$ short character sum. Here we are using the same strategy and ideas developed in the paper of R. Munshi (see \cite{Mun}). We will only present the case of holomorphic cusp forms for $SL(2,\mathbb{Z})$ as for the Maass forms one can see R. Munshi's paper (see \cite{Mun}) which carried out the Maass form case in details. The case for Maass forms is just similar as we only need the Ramanujan bound in the $L^2$-sense.

\begin{remark}
 Here we are considering $p$ to be a prime number for simplicity but also one can do for $p$ when p is not a prime (one has to handle coprimality issue carefully) using the same method.
\end{remark}

 \subsection{Sketch of the proof} We shall describe our method briefly by taking $f$ to be a holomorphic primitive cusp form for $SL(2,\mathbb{Z})$. Here we are using the method of R. Munshi (see \cite{Mun}). At first we consider the sum 
 $$\mathbf{S}:= \mathop\sum_{n \sim N}\lambda_{f}(n)\chi(n),$$
 for $N > p^{\frac{2}{3}+\epsilon}$ for some $\epsilon > 0$, where $\lambda_{f}(n)$'s are the normalized Fourier coefficients of $f$ and $p$ is the conductor of $\chi$. Here in the sketch we will suppress the weight function for notational simplicity. Then we write this sum as 
 $$\mathbf{S}= \mathop{\sum\sum}_{n,m \sim N}\lambda_{f}(n)\chi (m)\delta_{n,m},$$
 where $\delta_{n,m}$ is the Kronecker $\delta$-symbol. Here to get an inbuilt bilinear structure in the circle method itself, we need to use a more flexible version of the circle method - the one investigated by M. Jutila (see \cite{PA}, \cite{MZ}). This version comes with an error term which is satisfactory, as we shall find out, as long as we allow the moduli to be slightly larger than $\sqrt{N}$ (see Section \ref{sec3}). Up to an admissible error we see that $\mathbf{S}$ is given by
 $$\mathbf{S}=\mathop{\sum\sum}_{n,m \sim N}\lambda_{f}(n)\chi (m)\int_{\mathbb{R}}\tilde{I}(\alpha )e((n-m)\alpha)d\alpha ,$$
 where $\widetilde{I}(\alpha ): = \frac{1}{2\delta L} \mathop\sum_{q\in \Phi}\sum_{d ( \mod q)}^{\star }I_{d/q}(\alpha )$ 
 and $I_{d/q}$ is the indicator function of the interval $[\frac{d}{q}-\delta ,\frac{d}{q} + \delta ]$, $Q:= N^{1/2 + \epsilon}$ and $L \asymp Q^{2-\epsilon}$ (see Subsection \ref{circle}).
 
Trivial bound at this stage yields $N^{2+\epsilon}$ and we need to establish the bound $N^{1-\theta}$ for some $\theta >0$, i.e., roughly speaking we need to save $N$. Observe that by our choice of $Q$, there is no analytic oscillation in the weight function $e((n-m)\alpha )$. Hence their weights can be dropped in our sketch. At first using the $GL(2)$ Voronoi summation formula on the $n$ sum we get that
 $$\sum_{n\sim N}\lambda_{f}(n)e\left(\frac{na}{q}\right) \approx \frac{N}{q}\sum_{n \sim \frac{Q^2}{N}}\lambda_{f}(n)e\left(\frac{-n\bar{a}}{q}\right),$$
 where $q$ is of size $Q= N^{1/2 + \theta}$. The left hand side is trivially bounded by $N$, whereas the right hand side is trivially bounded by $Q$. Hence we have ``saved" $\frac{N}{Q}= \sqrt{\frac{N}{N^{2\theta}}}$.
 
$\hspace{-0.39 cm}$ Now applying the Poisson summation formula to the $m$ sum we arrive at
 $$\sum_{m\sim N}\chi (m)e\left(-\frac{ma}{q}\right)\approx \frac{N \tau_{\chi}}{p}\sum_{|m|\ll \frac{pQ}{N}}\bar{\chi}(m)\chi(q)\mathbbm{1}_{a \equiv m\bar{p} (\text{mod }q)},$$
 where $\mathbbm{1}_{a \equiv m\bar{p} (\text{mod }q)}$ is the indicator function for $a \equiv m\bar{p} (\text{mod }q)$ on $\mathbb{Z}$. Compairing the trivial contribution of the two sides we observe that we have ``saved"
 $$\frac{N}{\sqrt{pQ}}\times \sqrt{Q}= \frac{N}{\sqrt{p}}.$$
 
$\hspace{-0.4 cm}$ With this the above sum roughly reduce to
$$\mathbf{S} \approx \frac{N^2}{Q^3 p^{1/2}}\sum_{q \in \Phi}\sum_{n \sim N^{2\theta}}\sum_{m \sim \frac{p N^\theta}{\sqrt{N}}}\lambda_{f}(n)\bar{\chi}(m)\chi(q)e\left(-\frac{\bar{m}np}{q}\right).$$
 
 $\hspace{-0.39 cm}$So far we have ``saved" $N^{1/2 - \theta}\times \frac{N}{\sqrt{p}}=\frac{N^{3/2 - \theta}}{\sqrt{p}}.$ Hence our job is to ``save" $\frac{N}{\frac{N^{3/2 - \theta}}{\sqrt{p}}}=\frac{\sqrt{p}}{N^{1/2 -\theta}}$ in the above sum.

 Next we choose $Q= Q_1 Q_2$ and take the set of moduli $\Phi$  to be  a product of two sets of primes so that (as was done in the Subsections \ref{circle}, \ref{sec5}) $q = q_1 q_2$ in a certain unique way with $q_1 \leq Q_1$ and $q_2 \leq Q_2$ (see Section \ref{sec5}). Then applying the Cauchy-Schwarz inequality we arrive at
 $$\sum_{q_1 \sim Q_1}\sum_{m \sim \frac{p N^\theta}{\sqrt{N}}}\Big| \sum_{n \sim N^{2\theta}}\sum_{q_2 \sim Q_2}\lambda_{f}(n)\chi(q_2 )e\left(-\frac{\bar{m}np}{q_1 q_2}\right)\Big|^{2}.$$
 Next we open the absolute value square and apply the Poisson summation formula to the $m$-sum (after appropriate smoothing). Here the diagonal is of length $Q_2 N^{2\theta}$ and so the contribution of the zero frequency is given by $\ll p N^{4\theta}$. Hence the diagonal contribution is satisfactory if
 $$Q_2 N^{2\theta}> \frac{p}{N^{1-2\theta}}\text{ , i.e., } Q_2 > \frac{p}{N}.$$
  Also the contribution of the off-diagonal is given by $\ll p N^{6 \theta}$.
 Note that this is satisfactory if 
 $$\frac{p N^{\theta /2}}{N^{3/4}\sqrt{Q_2}} > \frac{p}{N^{1-2\theta}} \iff Q_2 < N^{1/2 -3\theta}.$$
 So we have a choice for $Q_2$ if
 $$\frac{p}{N}< N^{1/2 - 3\theta} \Rightarrow p < N^{3/2 - 3\theta}.$$
 Hence as long as $N > p^{2/3 + \epsilon}$ for some $\epsilon> 0$ then the above method yields a non-trivial bound for $\mathbf{S}$.
 
\subsection*{Notation}
In this article `$\ll$' will mean that whenever it occurs, the implied constants will depend on $f, \epsilon$ only and the notation `$X\asymp Y$' will mean that $Y p^{-\epsilon}\leq X \leq Y p^{\epsilon}$.

\subsection*{Acknowledgement}
This work is a part of the author's Ph.D thesis and he is grateful to his advisor Prof. Ritabrata Munshi for suggesting the problem, sharing his beautiful ideas, explaining his ingenious method, and his kind support and encouragement throughout the work. The author is also thankful to Prof. Djordje Milićević for his helpful comments. The author is also thankful to Prof. Satadal Ganguly, Prof. Saurabh Kumar Singh, Kummari Mallesham, Sumit Kumar,  and Prahlad Sharma for their constant support and encouragement and Stat-Math Unit, Indian Statistical Institute, Kolkata, for the excellent research environment. Finally, the author would like to thank the referee for his/her comments and suggestions which really helped to improve the presentation of the article.
 
 \section{PRELIMINARIES}\label{sec2}
 
 \subsection{Preliminaries on holomorphic cusp forms.}  Let $f : \mathbb{H} \mapsto \mathbb{C}$, be a holomorphic cusp form with normalized Fourier coefficients $\lambda_{f}(n)$. Also we take $\chi$, a primitive Dirichlet character of modulus $p$ where $p$ is a prime.

 \subsection{Voronoi summation formula.}  We will use the following Voronoi summation formula. This was first established by T. Meurman (see \cite{TM}) in the case of full level.

 \begin{lemma}\label{2.1} Let f be as above, and $v$ be a compactly supported smooth function on $(0,\infty )$. Also consider $(a,q)=1$. Then we have
 \begin{equation}\label{eq9}
 \mathop\sum_{n=1}^{\infty}\lambda_{f}(n)e_{q}(an)v(n) = \frac{1}{q}\mathop\sum_{n=1}^{\infty}\lambda_{f}(n)e_{q}(- \bar{a}n)V(n),
 \end{equation}
 where $\bar{a}$ is the multiplicative inverse of $a\text{ mod q}$, and $V(n)$ is a certain integral Hankel transform of $v$.
\end{lemma} 

Here note that, if we take $v$ to be supported in $[Y, 2Y]$ and satisfying $y^j v^{(j)}(y) \ll_{j} 1$, then one can see that the sum on the right hand side of \eqref{eq9} becomes being supported essentially on $n \ll {q^2 (qY)^\epsilon /Y}$ (the implied constant depends only on $f \text{ and }\epsilon$). Also note that the terms with $n \gg {q^2 (qY)^\epsilon /Y}$ contributes an amount which is negligibly small. For smaller values of $n$ one can consider the trivial bound $V(n/ q^2)\ll Y$. For more details one can see the paper of R. Munshi (see \cite{Mun}).
 
 \subsection{Circle Method :}\label{circle}
 
 Here in this paper we shall use Jutila's circle method (see \cite{PA}, \cite{MZ}). For any set $S \subset R$, let $I_{S}$ denote the associated characteristic function, i.e. $I_{S}(x)=1$ for $x \in S$ and $0$ otherwise. For any collection of positive integers $\Phi \subset [Q,2Q]$ (which we call the set of moduli), where $Q \geq 1 $ and a positive real number $\delta$ in the range $Q^{-2} \ll \delta \ll Q^{-1}$, we define the function
$$\tilde{I}_{\Phi , \delta}(x):= \frac{1}{2\delta L}\sum_{q \in \Phi}\sum_{d \mod q}I_{[\frac{d}{q}-\delta , \frac{d}{q}+ \delta]}(x),$$
where $I_{[\frac{d}{q}-\delta , \frac{d}{q}+ \delta]}$ is the indicator function of the interval $[\frac{d}{q}-\delta , \frac{d}{q}+ \delta]$. Here $L := \mathop\sum_{q \in \Phi}\phi (q)$ (then roughly we have $L \asymp Q^2$) and we will choose $\Phi$ in such a way that $L \asymp Q^{2-\epsilon}$.
 
Then this becomes an approximation of $I_{[0,1]}$ in the following sense:
\begin{lemma}\label{2.2} We have
$$\int_{\mathbb{R}}\left| I_{[0,1]}(x) - \tilde{I}_{\Phi , \delta}(x)\right|^2 dx \ll \frac{Q^{2+\epsilon}}{\delta L^2},$$
where $I$ is the indicator function of $[0,1]$.
\end{lemma}

$\hspace{-0.39 cm}$ This is a consequence of the Parseval theorem from Fourier analysis (see \cite{PA}).
 
 \section{Setting-up the circle method :}\label{sec3}

$\hspace{-0.4 cm}$Let us apply the circle method directly to the smooth sum
$$S(N)= \mathop\sum_{n \in \mathbb{Z}}\lambda_{f}(n)\chi(n)h_{1}\left(\frac{n}{N}\right),$$
 where the function $h_1$ is smooth, supported in $[1,2]$ with $h_{1}^{(j)}(x)\ll_{j} 1$. Now we will approximate the above sum $S(N)$ using M. Jutila's circle method (see \cite{PA}, \cite{MZ}) by the following sum :
$$\tilde{S}(N)=\frac{1}{L}\mathop\sum_{q \in \Phi} \hspace{0.1cm} \sideset{}{^*}\sum_{a(\text{ mod } q)}\mathop{\sum\sum}_{n,m \in \mathbb{Z}}\lambda_{f}(n)\chi (m)e\left(\frac{a(n-m)}{q}\right)F(n,m),$$
 where $e_{q}(x)= e^{2\pi i x /q}$, and 
$$F(x,y)= h_{1}\left(\frac{x}{N}\right)h_{2}\left(\frac{y}{N}\right)\frac{1}{2\delta}\int_{-\delta}^{\delta}e(\alpha (n-m))d\alpha .$$
 Here $h_{2}$ is another smooth function having compact support in $(0,\infty)$, with $h_{2}(x)=1$ for $x$ in the support of $h_{1}$. Also we choose $\delta = N^{-1}$ so that we have 
 $$\frac{\partial^{i+j}}{\partial^{i}x \partial^{j}y}F(x,y)\ll_{i,j}\frac{1}{N^{i+j}}.$$

$\hspace{-0.4 cm}$Then we have the following lemma :

\begin{lemma}\label{3.1} Let $\Phi \subset [Q,2Q]$, with 
$$L=\sum_{q\in \Phi}\phi(q) \gg Q^{2-\epsilon},$$
 and $\delta = \frac{1}{N} \gg \frac{N^{2\theta}}{Q^2}$. Then we must have 

$$S(N)= \tilde{S}(N) + O_{f,\epsilon}\left( \sqrt{N} \frac{N(QN)^{\epsilon}}{Q}\right).$$
\end{lemma}

\begin{proof}
For the proof of this lemma one can see Lemma $3$ of \cite{Mun}.
\end{proof}

$\hspace{-0.4 cm}$We will choose the set of modulii in Section \ref{sec5}. We shall pick the set of the modulii to be $Q = N^{1/2 +\theta }$. Hence the error term getting from the previous lemma is $O(N^{1-\theta +\epsilon})$ for some $\theta > 0$. Now we shall proceed towards the estimation of $\tilde{S}(N)$.

\section{Estimation of $\tilde{S}(N)$}\label{sec4}
\subsection{Applying summation formulae}
Let us now assume that each member of $\Phi$ is coprime to $p$, the modulus of the character $\chi$. Let us define
\begin{equation}\label{eq10}
\tilde{S}_{x}(N) =\frac{1}{L}\mathop\sum_{q\in \Phi} \hspace{2mm}\sideset{}{^*}\sum_{a(\text{mod }q)} S\left(a,q,x,f\right) \, T\left(a,q,x,\chi\right),
\end{equation}
where 
$$S\left(a,q,x,f\right):= \sum_{n \in \mathbb{Z}}\lambda_{f}(n)h_{1}\left(\frac{n}{N}\right)e\left(\frac{ an}{q}\right)e(nx),$$ 
 and 
$$T\left(a,q,x,\chi\right):=  \sum_{m \in \mathbb{Z}}\chi (m)e\left(-\frac{am}{q}\right)h_{2}\left(\frac{m}{N}\right)e(-xm) ,$$

$\hspace{-0.39 cm}$with $|x|< \delta$. Then we have 
$$\tilde{S}(N)= (2\delta)^{-1}\int_{-\delta}^{\delta}\tilde{S}_{x}(N) dx .$$

$\hspace{-0.39 cm}$Let us first study the $n$-sum using the Voronoi summation formula.
\begin{equation}\label{eq11}
S\left(a,q,x,f\right)= \sum_{n =1}^{\infty}\lambda_{f}(n)h_{1}\left(\frac{n}{N}\right)e\left(\frac{ an}{q}\right)e(nx).
\end{equation}

$\hspace{-0.39 cm}$Then we have the following lemma: 

\begin{lemma}\label{4.1} We have
\begin{equation}\label{eq12}
S\left(a,q,x,f\right)= \frac{N^{3/4}}{q^{1/2}}\sum_{|n|\ll \frac{Q^2}{N}}\frac{\lambda_{f}(n)}{n^{1/4}}e\left(-\frac{\bar{a}n}{q}\right)\mathcal{I}_{1}(n,x,q) + O(N^{-2021}),
\end{equation}
where $q \in [Q, 2Q]$, coprime with $p$ and $\mathcal{I}_{1}(n,x,q)$ is given by
$$\mathcal{I}_{1}(n,x,q):= \int_{\mathbb{R}}h_{1}(y)e\left( Nxy \pm \frac{4\pi}{q}\sqrt{Nny}\right) W\left(\frac{4\pi \sqrt{Nny}}{q}\right) dy ,$$
where $W$ is a smooth nice function.

\end{lemma}
\begin{proof}
Applying the Voronoi summation formula \eqref{eq9} to the $n$-sum of the equation \eqref{eq11}, then we have
$$\hspace{-4.6cm}\sum_{n \in \mathbb{Z}}\lambda_{f}(n)e\left(\frac{an}{q}\right)e(nx)h_{1}\left(\frac{n}{N}\right) = \frac{1}{q}\sum_{n \in \mathbb{Z}}\lambda_{f}(n)e\left(-\frac{\bar{a}n}{q}\right)$$
 $$\hspace{5cm}\times\int_{\mathbb{R}}h_{1}\left(\frac{y}{N}\right)e(xy)J_{k-1}\left(\frac{4\pi \sqrt{ny}}{q}\right) dy, $$
 where $J_{k-1}$ is the Bessel function. By changing $y \mapsto Ny$ and using the decomposition, 
 $$J_{k-1}(x)= \frac{W(x)}{\sqrt{x}}e(x)+ \frac{\bar{W}(x)}{\sqrt{x}}e(-x) ,$$
  where $W(x)$ is a nice function, the right hand side integral becomes
 $$N^{3/4}q^{1/2} \int_{\mathbb{R}}h_{1}(y)e\left( Nxy \pm \frac{4\pi}{q}\sqrt{Nny}\right) W\left(\frac{4\pi \sqrt{Nny}}{q}\right) dy . $$
By repeated integral by parts we see that, this integral is negligibly small if $|n|\gg \frac{Q^{2} N^\epsilon}{N}$. Hence the lemma follows. 
 \end{proof}
 \begin{remark}
 Note that $x \asymp \frac{\sqrt{n}}{\sqrt{N}q}$, otherwise $\mathcal{I}_{1}(n,x,q)$ is negligibly small.
 \end{remark}
$\hspace{-0.4 cm}$ Now let us consider the $m$-sum of \eqref{eq10} given by
 \begin{equation}\label{eq13}
 T\left(a,q,x,\chi\right)=  \sum_{m \in \mathbb{Z}}\chi (m)e\left(-\frac{am}{q}\right)h_{2}\left(\frac{m}{N}\right)e(-xm)
,
 \end{equation}
 for which we have the following lemma:
 \begin{lemma}\label{4.2}We have
  \begin{equation}\label{eq14}
   T\left(a,q,x,\chi\right) = \frac{N\tau_{\chi}}{p}\sum_{\substack{|m| \ll \frac{pQ}{N} \\ m\bar{p}\equiv a (\text{mod }q)}}\overline{\chi}(m)\chi(q) \mathcal{I}_{2}(m,x,q) + O(N^{-2021}),
  \end{equation}
  where 
  $$\mathcal{I}_{2}(m,x,q) := \int_{\mathbb{R}}h_{2}(y)e(-Nxy)e\left(-\frac{mNy}{pq}\right) dy .$$
 \end{lemma}

\begin{proof}
To the $m$-sum in the equation \eqref{eq13}, we apply the Poisson summation formula to get that
$$T\left(a,q,x,\chi\right) = \frac{N}{pq}\sum_{m \in \mathbb{Z}}\mathcal{I}_{2}(m,x,q)\sum_{\beta \text{( mod }q)}\chi (\beta)e\left(-\frac{a\beta}{q}\right) e\left(\frac{m\beta}{pq}\right) ,$$
where 
$$\mathcal{I}_{2}(m,x,q) := \int_{\mathbb{R}}h_{2}(y)e(-Nxy)e\left(-\frac{mNy}{pq}\right) dy .$$
 Here note that, this integral is negligibly small if $|m|\gg \frac{pQ}{N}N^\epsilon$.

So we have
$$T\left(a,q,x,\chi\right) = \frac{N}{pq}\sum_{|m|\ll \frac{pQ}{N}N^\epsilon}\mathcal{I}_{2}(m,x,q)\sum_{\beta \text{( mod }pq)}\chi (\beta)e\left(-\frac{a\beta}{q}\right) e\left(\frac{m\beta}{pq}\right) + O(N^{-2021}).$$
 
 As we know $(p,q)=1$, so that we can write $\beta$ as $\beta = \beta_{1}q\bar{q}+\beta_{2}p\bar{p}$, where $\beta_{1},\beta_{2}$ runs through a complete set of residue classes congruent to $p,q$ respectively.
 Then substituting these in the place of $\beta$ we have
 $$T\left(a,q,x,\chi\right) = \frac{N\tau_{\chi}}{p}\sum_{\substack{|m| \ll \frac{pQ}{N} \\ m\bar{p}\equiv a (\text{ mod }q)}}\overline{\chi}(m)\chi(q) \mathcal{I}_{2}(m,x,q)+ O(N^{-2021}).$$
 
$\hspace{-0.45 cm}$This completes the proof.
\end{proof}

$\hspace{-0.39 cm}$From \eqref{eq12} and \eqref{eq14} we get:

\begin{proposition}\label{p1}
We have
\begin{equation}\label{eq15}
\begin{split}
\tilde{S}_{x}(N) &= \frac{N^{7/4}}{\sqrt{p}L}\sum_{q\in \Phi}\frac{\chi (q)}{q^{1/2}}\sum_{|n| \ll \frac{Q^{2}}{N}}\sum_{\substack{|m| \ll \frac{pQ}{N} \\ (m,q)=1}}\frac{\lambda_{f}( n)}{n^{1/4}}\overline{\chi}(m)e\left(-\frac{p\bar{m}n}{q}\right)\mathcal{I}_{1}(n,x,q)\mathcal{I}_{2}(m,x,q)  \\
&  \hspace{10cm} +  O(N^{-2021}),
\end{split}
\end{equation}
where $\mathcal{I}_{1}(n,x,q),\mathcal{I}_{2}(m,x,q)$ are given by \eqref{eq12},\eqref{eq14} respectively.
\end{proposition}

\section{Further Estimation}\label{sec5}
\subsection{Applying the Cauchy and Poisson summation formulae} We choose the set of moduli $\Phi$ to be the product set $\Phi_{1}\Phi_{2}$, where $\Phi_{i}$ consists of primes in the dyadic segment $[Q_{i}, 2Q_{i}]$ (and not dividing $p$) for $i = 1, 2$ and $Q_1 Q_2 = Q = N^{1/2 + \theta}$. Also we pick $Q_1$ and $Q_2$ (whose optimal sizes will be determined later) so that the collections $\Phi_1$ and $\Phi_{2}$ are disjoint. Now consider $M_0 := \frac{pQ}{N}, N_0 := \frac{Q^2}{N}$. Here we note that, as $0<\theta < 1/10$ so that we have $Q_2 > N_0$ and also we have $Q_1 > N_0$. 

 Now applying the Cauchy-Schwarz inequality to the equation \eqref{eq15}, we arrive at
 \begin{equation}\label{eq16}
 \begin{split}
 \tilde{S}_{x}(N) \ll \frac{N^{7/4}\sqrt{M_0}}{\sqrt{p}L \sqrt{ Q_1 }} \sum_{q_1 \in \Phi_{1}} & \\
 &\hspace{-4cm} \times\left(\sum_{|m|\ll M_0}\left| \sum_{q_2 \in \Phi_2}\frac{\chi(q_2)}{q_{2}^{1/2}}\sum_{|n|\ll N_0}\frac{\lambda_{f}(n)}{n^{1/4}}\mathcal{I}_{1}(n,x,q_1 q_2)\mathcal{I}_{2}(m,x,q_1 q_2) e\left(-\frac{p\bar{m}n}{q_1 q_2}\right)\right|^2 \right)^{1/2} \\
 & \hspace{-4cm}\ll \frac{N^{7/4}\sqrt{M_0}}{\sqrt{p}L\sqrt{Q_1}}\sum_{q_1 \in \Phi_1}\Omega( N_0 , q_1 , Q_2 ,x)^{1/2},
 \end{split}
 \end{equation}

$\hspace{-0.39 cm}$where $\Omega ( N_0 , q_1 , Q_2 ,x)$ is defined as
\begin{equation}\label{sigma}
\sum_{\substack{|m|\ll M_0 \\ (m,q)=1}} \left|\sum_{q_2 \in \Phi_2}\frac{\chi(q_2)}{q_{2}^{1/2}}\sum_{|n|\ll N_0}\frac{\lambda_{f}(n)}{n^{1/4}}\mathcal{I}_{1}(n,x,q_1 q_2)\mathcal{I}_{2}(m,x,q_1  q_2) e\left(-\frac{p\bar{m}n}{q_1 q_2}\right)\right|^2.
 \end{equation}
 
 $\hspace{-0.38 cm}$Now we apply the Poisson summation formula to the $m$-sum with the modulus $q_1 q_2 q_{2}^{\prime}$ in the equation \eqref{sigma}. To this end, we first split the sum over $m$ into dyadic blocks $m\sim M_1$, $M_1 \ll M_0$ and then opening the absolute value square in the equation \eqref{sigma}, we get that,
$$\Omega ( N_0 , q_1 , Q_2 ,x) = \sum_{q_2 , q_{2}^{\prime}\in \Phi_2}\frac{\chi (q_2 \bar{q_{2}^{\prime}})}{(q_2 q_{2}^{\prime})^{1/2}}\sum_{|n|,|n|^\prime \ll N_0}\frac{\lambda_{f}(n)\lambda_{f}(n^\prime )}{(n n^\prime )^{1/4}}\mathcal{I}_{1}(n,x,q_1 q_2) \overline{\mathcal{I}_{1}(n^\prime ,x,q_1 q_{2}^{\prime})} \Delta ,$$
where 
$$\Delta = \sum_{M_1 \ll M_0}\sum_{m \in \mathbb{Z}}W^{\prime}\left(\frac{m}{M_1}\right)  e\left(\frac{\bar{m}p(q_{2}^{\prime}n - n^\prime q_2)}{q_1 q_2 q_{2}^{\prime}}\right) \mathcal{I}_{2}(m,x,q_1 q_2) \overline{\mathcal{I}_{2}(m,x,q_1 q_{2}^{\prime})},$$
$W^{\prime}(x)$ is a non-negative smooth function supported on $[ 2/3 , 3]$ with $W^{\prime}(x) = 1$ for $x \in [1,2]$ and $W^{\prime (j)}(x)\ll_j 1$.
 
 $\hspace{-0.39 cm}$Now applying the Poisson summation formula to the $m$-sum it transforms into
 $$\frac{M_1}{q_1 q_2 q_{2}^{\prime}}\sum_{m \in \mathbb{Z}}S(p(q_{2}^{\prime}n - n^\prime q_2),m; q_1 q_{2}q_{2}^{\prime})\mathcal{I}(m,x,q_1 ,q_2 ,q_{2}^{\prime}),$$
 where 
 $$\mathcal{I}(m,x,q_1 ,q_2 ,q_{2}^{\prime}):=\int_{\mathbb{R}}\, W^{\prime}(y)\, \mathcal{I}_{2}(M_1 y,x, q_1 q_2) \, \overline{\mathcal{I}_{2}(M_1 y , x , q_1 q_{2}^{\prime})}\, e\left(-\frac{m M_1 y }{q_1 q_2 q_{2}^{\prime}}\right) \, dy.$$
  Here note that the integral $\mathcal{I}$ is negligilbly small if $|m|\gg \frac{Q_1 Q_{2}^{2}}{M_1}N^\epsilon =\frac{Q_2 Q}{M_1}N^\epsilon$. 
  
  $\hspace{-0.39 cm}$Let $R_1 =\frac{Q_2 Q}{M_1}$. So we get
\begin{equation}\label{eq17}
\tilde{S}_{x}(N) = \frac{N^{7/4}\sqrt{M_0}}{\sqrt{p}L\sqrt{Q_1}}\sum_{q_1 \in \Phi_1}\Omega ( N_0 , q_1 , Q_2 , x)^{1/2} + O(N^{-2021}),
 \end{equation}
 where

 \begin{multline*}
 \hspace{-0.25cm}\Omega( N_0 , q_1 , Q_2 ,x)= \sum_{M_1 \ll M_0} \frac{M_1}{q_1 }\sum_{q_2 , q_{2}^{\prime} \in \Phi_2}\frac{\chi (q_2 \bar{q_{2}^{\prime}})}{(q_2 q_{2}^{\prime})^{3/2}}\sum_{|n|,|n|^\prime \ll N_0}\frac{\lambda_{f}(n)\lambda_{f}(n^\prime )}{(n n^\prime )^{1/4}}\\
 \times \mathcal{I}_{1}(n,x,q_1 q_2) \,  \overline{\mathcal{I}_{1}(n^{\prime},x,q_1 q_{2}^{\prime})} \sum_{|m|\ll R_1 }S(p(q_{2}^{\prime}n - n^\prime q_2),m; q_1 q_{2}q_{2}^{\prime})\mathcal{I}(m,x,q_1 ,q_2 ,q_{2}^{\prime}).
 \end{multline*}

\begin{remark}\label{re1}
Let us recall the Rankin-Selberg bound for Fourier coefficients. If $\lambda_{f}(n)$ be the normalised Fourier coefficients of a holomorphic cusp form, or of a Maass form $f$. Then for any real number $x \geq 1$, we have 
$$\mathop{\sum}_{1\leq n \leq x}\left| \lambda_{f}(n) \right|^2 \ll_{f}x .$$
Moreover, by the work of Deligne (see \cite{DEL}) and Deligne–Serre (see \cite{DP}) (the latter is for $k=1$), the Ramanujan conjecture for holomorphic cusp forms is now well-known:
$$\lambda_{f}(n)\ll n^\epsilon .$$
 \end{remark}
 
 \begin{lemma}\label{5.1}
 We have 
 \begin{equation}\label{eq18}
 \Omega( N_0 , q_1 , Q_2 ,x) \ll M_0 N_{0}^{1/2} + N_{0}^{3/2}Q_{2}^{2} \sqrt{Q_1}.
 \end{equation}
 
 \end{lemma}
 The proof of this lemma is given below. The first term of right hand side of \eqref{eq18} is coming from $m=0$ and the second term is coming from other $m$'s, i.e., for the terms with $m\neq 0$. For the proof, at first we consider the zero frequency case, i.e., when $m=0$.
 
 \subsection*{The zero frequency}The zero frequency $m=0$ has to be treated differently. Let $\Sigma_0$ denote the contribution of the zero frequency to $\tilde{S}_{x}(N)$, i.e.,
\begin{equation}\label{eq19}
\begin{split}
\Sigma_0 =\sum_{M_1 \ll M_0}\frac{M_1}{q_1 }\sum_{q_2 , q_{2}^{\prime} \in \Phi_2}\frac{\chi (q_2 \bar{q_{2}^{\prime}})}{(q_2 q_{2}^{\prime})^{3/2}} \sum_{|n|,|n|^\prime \ll N_o}\frac{\lambda_{f}(n)\lambda_{f}(n^\prime )}{(n n^\prime )^{1/4}} S(p(q_{2}^{\prime}n - n^\prime q_2),0; q_1 q_{2}q_{2}^{\prime})\\
& \hspace{-12cm}\times \mathcal{I}_{1}(n,x,q_1 q_2) \,  \overline{\mathcal{I}_{1}(n^{\prime},x,q_1 q_{2}^{\prime})} \, \mathcal{I}(0,x,q_1 ,q_2 ,q_{2}^{\prime})\\
& \hspace{-14cm}=\sum_{M_1 \ll M_0} \frac{M_1}{q_1 }\sum_{q_2 , q_{2}^{\prime} \in \Phi_2}\frac{\chi (q_2 \bar{q_{2}^{\prime}})}{(q_2 q_{2}^{\prime})^{3/2}}\sum_{|n|,|n|^\prime \ll N_0}  \frac{\lambda_{f}(n)\lambda_{f}(n^\prime )}{(n n^\prime )^{1/4}}\\
& \hspace{-14cm}\times \sum_{d | ( q_1 q_2 q_{2}^{\prime},p(q_{2}^{\prime}n - n^{\prime}q_2 ))}d \mu \left(\frac{p(q_{2}^{\prime}n - n^{\prime}q_2 )}{d}\right)\mathcal{I}_{1}(n,x,q_1 q_2) \,  \overline{\mathcal{I}_{1}(n^{\prime},x,q_1 q_{2}^{\prime})} \, \mathcal{I}(0,x,q_1 ,q_2 ,q_{2}^{\prime})\\
&\hspace{-14cm}=\sum_{M_1 \ll M_0}\frac{M_1}{q_1 }\sum_{q_2 , q_{2}^{\prime} \in \Phi_2}\frac{\chi (q_2 \bar{q_{2}^{\prime}})}{(q_2 q_{2}^{\prime})^{3/2}}\sum_{d | q_1 q_2 q_{2}^{\prime}}d \sum_{\substack{|n|,|n^{\prime}|\ll N_0 \\ q_{2}^{\prime}n - n^{\prime}q_2 \equiv 0 (\text{mod d})}}\frac{\lambda_{f}(n)\lambda_{f}(n^\prime )}{(n n^\prime )^{1/4}}\\
& \hspace{-12cm}\times \mathcal{I}_{1}(n,x,q_1 q_2)\,  \overline{\mathcal{I}_{1}(n^\prime,x,q_1 q_{2}^{\prime})}\, \mathcal{I}(0,x,q_1 ,q_2 ,q_{2}^{\prime}).
\end{split}
\end{equation}
\begin{lemma}\label{5.2} We have
\begin{equation}\label{eq20}
\Sigma_0 \ll M_0 N_{0}^{1/2}.
\end{equation}
\end{lemma}

\begin{proof}
 For $m=0$ we have six cases according to the divisors of $q_1 q_2 q_{2}^{\prime}$ and note that $Q= Q_1 Q_2$.

\subsubsection*{Case 1 }Let $d=q_1 q_2 q_{2}^{\prime}$. Then note that size of $d$ for this case is $Q_1 Q_{2}^{2}$. But size of $q_{2}^{\prime}n - n^{\prime}q_2$ is $Q_2 N_0$. So for this case 
$$q_{2}^{\prime}n - n^{\prime}q_2 \equiv 0 (\text{mod } d) \iff q_2 = q_{2}^{\prime} \text{ and }n= n^{\prime},$$
 as size of $n$ is smaller than size of $Q_2$. Hence we have, using the well known pointwise Ramanujan bound, given in remark \ref{re1},
$$\Sigma_{0} \ll \sup_{M_1 \ll M_0}M_1 N_{0}^{1/2} ,$$
as there are atmost $\log M_0 \, (\ll p^\epsilon )$ many $M_1$'s. Hence we have
$$\Sigma_{0} \ll M_0 N_{0}^{1/2} .$$

\subsubsection*{Case 2}Let $d=1$. Then we get that, as done in the previous case,
$$\Sigma_{0} \ll \sup_{M_1 \ll M_0}\frac{M_1 N_{0}^{3/2}}{Q} \ll \frac{M_0 N_{0}^{3/2}}{Q}.$$
 But as $N_0 \ll Q$ so for this case, we must have,  
$$\Sigma_{0} \ll M_0 N_{0}^{1/2}.$$
\subsubsection*{Case 3}Let $d= q_1$. For this case, we have, 
$$\Sigma_{0}\ll \sup_{M_1 \ll M_0}\frac{M_{1}N_{0}^{3/2}}{Q_2} \ll \frac{M_{0}N_{0}^{3/2}}{Q_2}.$$
 But as $N_0 < Q_2$ so that for this case again we have,
 $$\Sigma_{0} \ll M_0 N_{0}^{1/2}.$$
\subsubsection*{Case 4}Now consider $d= q_2$. For this case we have,  
$$\Sigma_{0}\ll \sup_{M_1 \ll M_0}\frac{M_1 N_{0}^{3/2}}{Q_1} \ll \frac{M_0 N_{0}^{3/2}}{Q_1}.$$
 But as $N_0 < Q_1$ so for this case, we must have,  
 $$\Sigma_{0} \ll M_0 N_{0}^{1/2}.$$

$\hspace{-0.39 cm}$For the case $d= q_{2}^{\prime}$ we have to process similarly and we will get the same bound.
\subsubsection*{Case 5 }\label{5} Now take $d= q_1 q_2$. But as size of $q_{2}^{\prime}n - n^{\prime}q_2$ is $Q_2 N_0$ which is less than the size of $d$, i.e., $Q_1 Q_2$ so for this case we have 
$$q_{2}^{\prime}n - n^{\prime}q_2 \equiv 0 (\text{mod d}) \iff q_2 = q_{2}^{\prime} \text{ and }n= n^{\prime}.$$
 Hence we have 
 $$\Sigma_{0}\ll \sup_{M_1 \ll M_0}\frac{M_1 N_{0}^{1/2}}{Q_2} \ll \frac{M_0 N_{0}^{1/2}}{Q_2}.$$
  But then again we have, for this case, 
  $$\Sigma_{0} \ll M_0 N_{0}^{1/2} .$$
For $d= q_{1}q_{2}^{\prime}$ if we process similarly then we shall get the same bound.
\subsubsection*{Case 6 }For the last case we have $d= q_{2}^{2}$. This case will be similar as case $5$.  By considering the size of $d$ for this case again we can say that 
$$q_{2}^{\prime}n - n^{\prime}q_2 \equiv 0 (\text{mod d}) \iff q_2 = q_{2}^{\prime} \text{ and }n= n^{\prime}.$$
 So for this case we have 
 $$\Sigma_{0} \ll \sup_{M_1 \ll M_0}\frac{M_1 N_{0}^{1/2}}{Q_1}\ll \frac{M_0 N_{0}^{1/2}}{Q_1}.$$
  Hence we have, for this case, 
  $$\Sigma_{0} \ll M_0 N_{0}^{1/2}.$$
This completes the proof of the Lemma \ref{5.2}.
\end{proof}

\subsection*{Non-zero frequency}Now we will consider the non-zero frequency case, i.e., when $m\neq 0$. In this case we will need the following basic lemma:
\begin{lemma}\label{eq100}
For any $x,y \in \mathbb{R} \text{ with }x,y \geq 1$ and $c\in \mathbb{N}$, we have,
$$\sum_{1\leq a\leq x}\sum_{1\leq b\leq y}(a,b,c)\leq (xy)^{1+ \epsilon} .$$
\end{lemma}
\begin{proof}
Let $(a,b)=d$ so that $d \geq 1$. Then we have
$$\sum_{1\leq a\leq x}\sum_{1\leq b\leq y}(a,b,c)\leq \sum_{1\leq a\leq x}\sum_{1\leq b\leq y}(a,b) \leq \left( \sum_{1\leq d\leq x}d\sum_{1\leq b\leq \frac{y}{d}}1 \right) (xy)^{\epsilon} = (xy)^{1+ \epsilon} .$$
This completes the proof of this lemma.
\end{proof}
Here the contribution of the non-zero frequency to $\tilde{S}_{x}(N)$ is given by the following:
\begin{equation}\label{eq21}
\begin{split}
\Sigma_{\neq 0} =\sum_{M_1 \ll M_0}\frac{M_1}{q_1 }\sum_{q_2 , q_{2}^{\prime} \in \Phi_2}\frac{\chi (q_2 \bar{q_{2}^{\prime}})}{(q_2 q_{2}^{\prime})^{3/2}} \sum_{|n|,|n|^\prime \ll N_0}\frac{\lambda_{f}(n)\lambda_{f}(n^\prime )}{(n n^\prime )^{1/4}} \mathcal{I}_{1}(n,x,q_1 q_2) \overline{\mathcal{I}_{1}(n^\prime ,x,q_1 q_{2}^{\prime})}\\
& \hspace{-10cm}\times \sum_{1 \leq |m| \ll R_1}S(p(q_{2}^{\prime}n - n^\prime q_2),m; q_1 q_{2}q_{2}^{\prime})\mathcal{I}(m,x,q_1 ,q_2 ,q_{2}^{\prime}).
\end{split}
\end{equation}
By the Weil's bound for Kloosterman sums we arrive at
\begin{multline*}
\sum_{1\leq |m|\ll R_1 }S(p(q_{2}^{\prime}n - n^\prime q_2),m; q_1 q_{2}q_{2}^{\prime})\mathcal{I}(m,x,q_1 ,q_2 ,q_{2}^{\prime}) \\
 \ll Q_2 \sqrt{Q_1}\sum_{1\leq |m|\ll R_1}(p(q_{2}^{\prime}n - n^\prime q_2),m; q_1 q_{2}q_{2}^{\prime})^{1/2}.
\end{multline*}

$\hspace{-0.39 cm}$Then by the previous lemma, we have, 
$$\sum_{1 \leq |m|\ll R_1}(p(q_{2}^{\prime}n - n^\prime q_2),m; q_1 q_{2}q_{2}^{\prime})^{1/2} \ll R_{1}^{1+\epsilon} .$$
Hence we get that
\begin{equation}\label{22}
\sum_{|m|\ll R_1 }S(p(q_{2}^{\prime}n - n^\prime q_2),m; q_1 q_{2}q_{2}^{\prime})\mathcal{I}(m,x,q_1 ,q_2 ,q_{2}^{\prime})\ll R_{1}^{1+\epsilon}Q_2 \sqrt{Q_1}.
\end{equation}

$\hspace{-0.39 cm}$Now putting values of $R_1$ we get that, using the well known pointwise Ramanujan bound, given in Remark \ref{re1},
$$\Sigma_{\neq 0} \ll \sup_{M_1 \ll M_0}\frac{M_1}{Q_1}\times \frac{1}{Q_2}\times N_{0}^{3/2}\times \frac{Q_2 Q}{M_1}\times Q_2 \sqrt{Q_1},$$
as there are atmost $\log M_0 \, (\ll p^\epsilon )$ many $M_1$'s. So we have,
$$\Sigma_{\neq 0} \ll N_{0}^{3/2}Q_{2}^{2}\sqrt{Q_1}.$$

$\hspace{-0.39 cm}$This completes the proof of the Lemma \ref{5.1}.

\section{Final estimation}$\hspace{-0.39 cm}$From the equations \eqref{eq17} and \eqref{eq18} we get that
\begin{equation}\label{23}
\begin{split}
\tilde{S}_{x}(N) & \ll  \frac{N^{7/4}\sqrt{M_0}}{\sqrt{p}L\sqrt{Q_1}}\sum_{q_1 \sim Q_1}\left(  M_0 N_{0}^{1/2} + N_{0}^{3/2}Q_{2}^{2} \sqrt{Q_1} \right)^{1/2} \\
& \ll \frac{N^{7/4}\sqrt{Q_1 M_0}}{\sqrt{p}L }\left(  M_{0}^{1/2} N_{0}^{1/4} + N_{0}^{3/4}Q_{2} Q_{1}^{1/4} \right).
\end{split}
\end{equation}

Now the optimal choice of $Q_1$ is obtained by equating the two terms of the equation \eqref{eq18} and using the relations $Q_1 Q_2 =Q= N^{1/2 +\theta}$, $N_0 = \frac{Q^2}{N} \text{ and } M_0 = \frac{pQ}{N}$, so that we have
\begin{equation}\label{eq24}
Q_1 =\frac{N^{1+ 2\theta}}{p^{2/3}}.
\end{equation}
This satisfies our requirement that $\frac{p}{N}< Q_2 < N^{1/2 - 3\theta}$. Now putting this value of $Q_1$ in the equation \eqref{23}, we get that
$$\tilde{S}_{x}(N) \ll N^{3/4 + \theta /2}p^{1/6}.$$ 
Therefore we have
\begin{equation}\label{eq25}
S(N) \ll N^{3/4 + \theta /2}p^{1/6} + N^{1-\theta}
\end{equation}
with $0<\theta < 1/10$. We choose $\theta = \frac{1}{6}-\frac{\log p}{9 \log N}$
 so that 
 \begin{equation}\label{eq26}
 N^{3/4 + \theta /2}p^{1/6} = N^{1-\theta}
 \end{equation}
  Our $\theta $ will satisfy the condition $0< \theta < \frac{1}{10}$ if $p^{2/3 +\alpha + \epsilon}< N < p$ where $\alpha = \frac{4\theta}{1- 6\theta}$ which is fine.
This completes the proof of the Theorem \ref{th}.

\
{}

\end{document}